\documentclass{amsart}

\usepackage{mathtools}
\usepackage{amsmath}
\usepackage{amsthm}
\usepackage{amssymb}
\usepackage{enumitem}
\usepackage{float}
\usepackage{iftex}
\ifptex
  \usepackage[dvipdfmx,hidelinks]{hyperref}
\else
  \usepackage[hidelinks]{hyperref}
\fi
\usepackage{cleveref}

\numberwithin{equation}{section}

\theoremstyle{plain}
\newtheorem{thm}{Theorem}[section]
\newtheorem{prop}[thm]{Proposition}
\newtheorem{lem}[thm]{Lemma}
\newtheorem*{unnumberedlem}{Lemma}
\newtheorem{cor}[thm]{Corollary}
\crefname{thm}{Theorem}{Theorems}
\Crefname{thm}{Theorem}{Theorems}
\crefname{prop}{Proposition}{Propositions}
\Crefname{prop}{Proposition}{Propositions}
\crefname{lem}{Lemma}{Lemmas}
\Crefname{lem}{Lemma}{Lemmas}
\crefname{cor}{Corollary}{Corollaries}
\Crefname{cor}{Corollary}{Corollaries}

\theoremstyle{definition}
\newtheorem{defn}[thm]{Definition}
\newtheorem*{unnumbereddefn}{Definition}
\crefname{defn}{Definition}{Definitions}
\Crefname{defn}{Definition}{Definitions}

\theoremstyle{remark}
\newtheorem{rem}[thm]{Remark}
\crefname{rem}{Remark}{Remarks}
\Crefname{rem}{Remark}{Remarks}

\newcommand{\R}{\mathbb{R}}
\newcommand{\B}{\mathbb{B}}
\newcommand{\Hyp}{\mathbb{H}}
\newcommand{\Sph}{\mathbb{S}}
\newcommand{\Pyr}{\mathcal{P}}
\newcommand{\Ple}{\mathcal{P}_{\le 1}}
\newcommand{\RGauss}[2]{\Gamma^{#1}_{#2^2}(\Hyp)}
\newcommand{\WGauss}[2]{\Hyp\exp_o(\Gamma^{#1}_{#2^2})}
\newcommand{\eps}{\varepsilon}
\DeclareMathOperator{\diam}{diam}
\DeclareMathOperator{\PartDiam}{PartDiam}
\DeclareMathOperator{\supp}{supp}
\DeclareMathOperator{\Sep}{Sep}
\DeclareMathOperator{\ObsDiam}{ObsDiam}
\DeclareMathOperator{\Ric}{Ric}
\DeclareMathOperator{\Hess}{Hess}

\DeclareMathOperator{\arsinh}{arsinh}
\DeclareMathOperator{\artanh}{artanh}

\title[Radial hyperbolic measures]{Radial Hyperbolic Measures: Shell Geometry, Pyramid Limits, and Gaussian Phase Transitions}
\author{Shigeaki Yokota}
\date{}
\keywords{metric measure space, pyramid, hyperbolic space, radial measure, Bakry--\'Emery Ricci curvature, hyperbolic Gaussian}
\subjclass[2020]{Primary 53C23; Secondary 60B12, 28A33}

\begin{document}

\begin{abstract}
	In high-dimensional hyperbolic space, concentration of a radial measure near a shell need not determine the pyramid limit: angular concentration and hyperbolic expansion alter separation. An effective radius is the scale on which positive-mass sets actually separate, which the raw radius need not give. With vanishing rescaling and radial fluctuations, radius convergence gives weak convergence to the pyramid of all metric measure spaces with the resulting diameter bound. At modal radial Gibbs shells, exponential decay of intrinsic shell curvature and dimension-normalized tangential Bakry--\'Emery Ricci curvature recovers the radius. The Gaussian intrinsic to hyperbolic volume and that obtained by wrapping a Euclidean Gaussian have different critical orders. They are L\'evy below those orders, infinitely dissipate above them, and at criticality converge to the corresponding diameter-bounded pyramids.
\end{abstract}

\maketitle

\section{Introduction}\label{sec:introduction}

Gromov initiated the geometry of mm-spaces in order to treat high-dimensional
measure concentration while retaining both the distance and the measure
\cite[Introduction]{shioya2016mmg}.  A triple $(X,d_X,\mu_X)$ is an mm-space,
or simply $X$, if $(X,d_X)$ is a complete separable metric space and $\mu_X$
is a Borel probability measure
\cite[Definition~2.8]{shioya2016mmg}.  We say that $X$ dominates an mm-space
$Y$, and write $Y\prec X$, if there is a $1$-Lipschitz map from $X$ to $Y$
that pushes $\mu_X$ forward to $\mu_Y$.  Recording the full domination order
of $X$ gives the pyramid $\Pyr(X)$ of all mm-spaces dominated by $X$
\cite[Definition~2.10 and Section~6.1]{shioya2016mmg}.

Gromov proposed the compactification of the space of mm-spaces through the map
$X\mapsto\Pyr(X)$ \cite{gromov2007met}.  Shioya metrized its topology and
formulated the space of all pyramids as a compact metric space
\cite[Definition~4.5 and Theorem~4.6]{shioya2022sugaku}.  Convergence in
$\Box$ distance is not identical to measured Gromov--Hausdorff convergence,
but the two are closely related \cite[Remark~4.34]{shioya2016mmg}.  For
$D\ge0$, let $\mathcal P_{\le D}$ denote the pyramid of all mm-spaces of
diameter at most $D$.  Weak convergence of pyramids means sequential
Painlev\'e--Kuratowski convergence in the $\Box$-metric space of mm-spaces.
In mm-space theory, this is also called weak Hausdorff convergence
\cite[Definition~6.4]{shioya2016mmg}.  This notion allows limits governed only
by a diameter constraint, beyond the class of smooth spaces.

Even when a radial distribution concentrates at a single radius, that raw radius does not determine the pyramid limit.  Angular concentration on the high-dimensional sphere carries a square-root dimensional scale, and hyperbolic distance amplifies the resulting loss logarithmically.  Thus the raw radial scale and the scale on which positive-mass sets separate need not agree.  The effective radius corrects this discrepancy.

Fix an origin $o$ in the curvature $-1$ hyperbolic space $(\Hyp^n,d_{\Hyp})$.  Let $\mu_n$ be a Borel probability measure radial about $o$, put $X_n\coloneqq(\supp\mu_n,d_{\Hyp},\mu_n)$, and let $\mathbf R_n$ be the radius of a point with distribution $\mu_n$.  For a positive sequence $(s_n)$, write $s_nX_n\coloneqq(\supp\mu_n,s_nd_{\Hyp},\mu_n)$.  The effective radius of the shell of radius $\rho>0$ at this scale is
\begin{equation}\label{eq:discrete-radius}
	a_{n,\rho}\coloneqq s_n\log\frac{\sinh\rho}{\sqrt n}.
\end{equation}
For $a\in\R$, put $a_+\coloneqq\max\{a,0\}$, and set $(-\infty)_+\coloneqq0$.

The simplest discrepancy occurs when $\mu_n$ is uniform on the geodesic sphere of radius $\rho_n=\log n$ and $s_n=(\log n)^{-1}$.  Its radial width is exactly zero, and the raw scaled radius is $s_n\rho_n=1$, whereas $a_{n,\rho_n}\to1/2$.  Nevertheless,
\[ \Pyr(s_nX_n)\longrightarrow\mathcal P_{\le1}, \]
while the distances between any fixed finite number of independent sample points converge to $2$.

\begin{thm}[Radial Dirac concentration and the effective radius]
\label{thm:main}
	Let $s_n>0$ with $s_n\to0$, and suppose that deterministic radii
	$\rho_n>0$ satisfy
	\begin{equation}\label{eq:general-dirac-assumption}
		s_n(\mathbf R_n-\rho_n)\overset{\mathrm P}{\longrightarrow}0.
	\end{equation}
	If $a_{n,\rho_n}\to a\in[-\infty,+\infty)$, then
	\begin{equation}\label{eq:general-pyramid-limit}
		\Pyr(s_nX_n)\longrightarrow\mathcal P_{\le2a_+}
	\end{equation}
	weakly.

	For each $n$, let $H_n\subset\Hyp^n$ be a totally geodesic hyperplane
	through $o$, and let $u_{H_n}$ be signed distance from $H_n$, positive
	on one chosen side.  Then
	\begin{equation}\label{eq:general-signed-witness}
		(s_nu_{H_n})_*\mu_n
		\overset{\mathrm d}{\longrightarrow}
		\frac12\delta_{-a_+}+\frac12\delta_{a_+}.
	\end{equation}
\end{thm}

The signed-distance limit is nondegenerate precisely when $a>0$, in which case it is the two-point distribution $\frac12\delta_{-a}+\frac12\delta_a$.  If $a\le0$, then $a_+=0$, so the signed-distance limit reduces to the point mass $\delta_0$, consistently with the one-point pyramid $\mathcal P_{\le0}$.

Here $\sinh\rho_n$ is the shell's tangential hyperbolic expansion, $\sqrt n$ is the loss from angular concentration, and the logarithm of their ratio is the tangential scale that survives the rescaling.  In the usual situation $s_n\rho_n\to r$ and $s_n\log n\to\lambda$, the condition $a_{n,\rho_n}\to a$ is checked by subtracting half the dimensional logarithmic scale from the raw radius.

For a positive rescaling sequence $(\alpha_n)$ and fixed $\sigma>0$, the two Gaussian families have the following phase classification.
\begin{table}[H]
	\centering
	\small
	\setlength{\tabcolsep}{3pt}
	\begin{tabular}{@{}p{0.20\textwidth}p{0.09\textwidth}p{0.13\textwidth}p{0.23\textwidth}p{0.21\textwidth}@{}}
		\hline
		Measure family & Critical scale $t_n$ & $\alpha_n/t_n\to0$ & $\alpha_n/t_n\to c\in(0,\infty)$ & $\alpha_n/t_n\to+\infty$ \\
		\hline
		Riemannian\newline Gaussian & $n^{-1}$ & L\'evy & $\mathcal P_{\le2c\sigma^2}$ & Infinite dissipation \\
		Wrapped\newline Gaussian & $n^{-1/2}$ & L\'evy & $\mathcal P_{\le2c\sigma}$ & Infinite dissipation \\
		\hline
	\end{tabular}
	\caption{Phase classification of the two hyperbolic Gaussian families.}
	\label{tab:intro-gaussian-phases}
\end{table}
The critical scales differ because hyperbolic volume weighting in the Riemannian model and exponential-map wrapping in the second model produce different radial concentration scales before the angular correction is applied.

The proof first couples the radial measure to a deterministic shell at its concentration radius while keeping the same angular variable.  On that shell, hyperbolic distance is a monotone transform of the $\sqrt n$-scaled spherical chord distance, and separation of positive-mass sets identifies the diameter-bounded pyramid.

The radial decomposition supplies the deterministic-shell reduction, and the curvature results recover the effective radius as the exponential decay rate of the shell's curvatures.  The Gaussian radial limit theorems provide the hypotheses for the critical pyramid limits and the phase-transition conclusion.

\section{Radiality in the Poincar\'e ball}\label{sec:radiality}

Let
\[ \B^n\coloneqq\{x\in\R^n\mid |x|<1\} \]
carry the Poincar\'e metric
\[ g_x\coloneqq\frac4{(1-|x|^2)^2}g_{\mathrm E}. \]
Its distance is
\begin{equation}\label{eq:poincare-distance}
	d_{\Hyp}(x,y)
	=
	2\arsinh\frac{|x-y|}
	{\sqrt{(1-|x|^2)(1-|y|^2)}}.
\end{equation}
In particular,
\begin{equation}\label{eq:origin-radius}
	r(x)\coloneqq d_{\Hyp}(0,x)
	=2\artanh|x|
	=\log\frac{1+|x|}{1-|x|}.
\end{equation}
Let $\sigma_{n-1}$ be normalized surface measure on $\Sph^{n-1}$.

\begin{defn}[Radial measure]\label{def:radial}
	A Borel probability measure $\mu$ on $\B^n$ is \emph{radial about the
	origin} if
\[ Q_*\mu=\mu \qquad(Q\in O(n)). \]
\end{defn}

\begin{prop}[Equivalent formulations of radiality]
\label{prop:equivalent-radiality}
	Let $(\Hyp^n,o)$ be a pointed curvature $-1$ hyperbolic space, let
	$\iota\colon(\Hyp^n,o)\to(\B^n,0)$ be a pointed isometry, and identify
	$T_o\Hyp^n$ isometrically with $\R^n$.  For a Borel probability measure
	$\mu$ on $\Hyp^n$, the following conditions are equivalent.
	\begin{enumerate}[label=\textup{(\roman*)}]
	\item The measure $\mu$ is invariant under every hyperbolic isometry
	fixing $o$.
	\item The measure $\iota_*\mu$ is radial in the sense of
	\Cref{def:radial}.
	\item The measure $(\exp_o^{-1})_*\mu$ is invariant under the ordinary
	$O(n)$-action on $T_o\Hyp^n\simeq\R^n$.
	\end{enumerate}
	The condition is independent of the choices of $\iota$ and of the
	orthonormal identification of the tangent space.
\end{prop}

\begin{proof}
	The stabilizer of the origin in the isometry group of the Poincar\'e ball
	is the ordinary action of $O(n)$.  This proves the equivalence of
	\textup{(i)} and \textup{(ii)}.  For every isometry $g$ fixing $o$,
	\[
		g\circ\exp_o=\exp_o\circ(dg)_o,
	\]
	and the differentials $(dg)_o$ exhaust the orthogonal group of
	$T_o\Hyp^n$.  This proves the equivalence with \textup{(iii)}.  Two
	pointed Poincar\'e identifications or two orthonormal tangent-space
	identifications differ by an orthogonal transformation, which proves the
	last assertion.  This completes the proof.
\end{proof}

The two coordinate descriptions satisfy, for $v\ne0$,
\begin{equation}\label{eq:exp-poincare} \iota(\exp_ov)=\tanh(|v|/2)\frac v{|v|}. \end{equation}

\begin{prop}[Exact radial--shell decomposition]
\label{prop:exact-decomposition}
	For a Borel probability measure $\mu$ on $\B^n$, the following
	conditions are equivalent.
	\begin{enumerate}[label=\textup{(\roman*)}]
	\item The measure $\mu$ is radial.
	\item There is a unique probability measure $\nu$ on $[0,+\infty)$
	such that
	\begin{equation}\label{eq:disintegration}
		\mu=\Psi_*(\nu\otimes\sigma_{n-1}),
		\qquad
		\Psi(r,\theta)\coloneqq\tanh(r/2)\theta.
	\end{equation}
	\item A point with distribution $\mu$ can be realized as
	\begin{equation}\label{eq:random-decomposition}
		\mathbf X=\tanh(R/2)\Theta,
		\qquad
		R\mathbin{\perp\!\!\!\perp}\Theta,
		\qquad
		\Theta\sim\sigma_{n-1}.
	\end{equation}
	\end{enumerate}
	The measure $\nu$ is given by
	$\nu=r_*\mu=(2\artanh|\cdot|)_*\mu$.  If $\omega_{n,r}$ denotes
	normalized measure on the geodesic sphere of radius $r$, then
	\[
		\mu=\int_{[0,+\infty)}\omega_{n,r}\,\nu(dr).
	\]
\end{prop}

\begin{proof}
	Suppose that $\mu$ is radial, and put $\nu\coloneqq r_*\mu$.  Average a
	bounded Borel function over normalized Haar measure on $O(n)$.  The orbit
	measure on each nonzero Euclidean sphere is uniform, and the choice of
	angle at the origin does not change its image.  This gives
	\eqref{eq:disintegration}.  Conversely, the right-hand side of
	\eqref{eq:disintegration} is $O(n)$-invariant.  Pushing forward by $r$
	proves uniqueness of $\nu$.  The random-variable and shell-mixture
	formulations are equivalent forms of the same decomposition.
	This completes the proof.
\end{proof}

\section{Deterministic shells and their separation scale}
\label{sec:shell-reduction}

For $\rho>0$, define the uniform shell
\begin{equation}\label{eq:shell}
	\Sigma_{n,\rho}
	\coloneqq
	(\Sph^{n-1},d_{n,\rho},\sigma_{n-1}),
	\qquad
	d_{n,\rho}(\theta,\eta)
	\coloneqq
	d_{\Hyp}\bigl(\tanh(\rho/2)\theta,\tanh(\rho/2)\eta\bigr).
\end{equation}

\begin{unnumbereddefn}[Couplings]
Let $\mathcal T(\mu_X,\mu_Y)$ be the set of Borel probability measures on
$X\times Y$ whose marginals are $\mu_X$ and $\mu_Y$. An element
$\pi\in\mathcal T(\mu_X,\mu_Y)$ is called a \emph{coupling}, or a
\emph{transport plan}.
\end{unnumbereddefn}

\begin{unnumbereddefn}[Box distance {\cite[Proposition~4.4 and Theorem~1.1]{nakajima2022coupling}}]
For a closed set $S\subset X\times Y$, put
\[ \operatorname{dis}S\coloneqq \sup_{(x,y),(x',y')\in S} \left|d_X(x,x')-d_Y(y,y')\right| \]
and define the \emph{box distance} by
\[ \Box(X,Y)\coloneqq \inf_{\pi,S}\max\{1-\pi(S),\operatorname{dis}S\}, \]
where the infimum is taken over all $\pi\in\mathcal T(\mu_X,\mu_Y)$ and all
closed sets $S\subset X\times Y$. This is the coupling characterization of the
box distance. In particular, if
$\pi(S)\ge1-\delta$ and $\operatorname{dis}S\le\eta$, then
$\Box(X,Y)\le\max\{\delta,\eta\}$.
\end{unnumbereddefn}

\begin{prop}[Radial coupling]\label{prop:radial-coupling}
	For every $s_n>0$, $\rho_n>0$, and $\eps>0$,
	\begin{equation}\label{eq:box-bound}
		\Box(s_nX_n,s_n\Sigma_{n,\rho_n})
		\le
		\max\left\{2\eps,
		\mathbb P\bigl(s_n|\mathbf R_n-\rho_n|>\eps\bigr)\right\}.
	\end{equation}
	In particular,
	\begin{equation}\label{eq:scaled-shell-dirac}
		s_n(\mathbf R_n-\rho_n)\overset{\mathrm P}{\longrightarrow}0
	\end{equation}
	implies
	\[
		\Box(s_nX_n,s_n\Sigma_{n,\rho_n})\longrightarrow0.
	\]
\end{prop}

\begin{proof}
	Use \Cref{prop:exact-decomposition} to couple
	$\tanh(\mathbf R_n/2)\Theta_n$ and $\tanh(\rho_n/2)\Theta_n$ with the same
	direction.  Points on the same radial geodesic satisfy
	\[
		d_{\Hyp}\bigl(\tanh(r/2)\theta,\tanh(t/2)\theta\bigr)=|r-t|.
	\]
	Denote this coupling by $\pi_n$, and let $S_n$ be the closure of the set
	of coupled pairs satisfying $s_n|r-\rho_n|\le\eps$. The triangle
	inequality bounds the scaled distortion before taking the closure by
	$2\eps$, and continuity of the distances gives
	$\operatorname{dis}S_n\le2\eps$. Moreover,
\[ 1-\pi_n(S_n) \le\mathbb P\bigl(s_n|\mathbf R_n-\rho_n|>\eps\bigr). \]
	The coupling characterization of the box distance stated above proves
	\eqref{eq:box-bound}. This completes the proof.
\end{proof}

\begin{prop}[Exact shell distance and curvature]\label{prop:shell-geometry}
	For every $\rho>0$,
	\begin{equation}\label{eq:exact-shell-metric}
		d_{n,\rho}(\theta,\eta)
		=
		2\arsinh\left(\frac{\sinh\rho}{2}|\theta-\eta|\right).
	\end{equation}
	The induced Riemannian metric and intrinsic sectional curvature of the
	shell are
	\begin{equation}\label{eq:induced-shell-metric}
		g_{\Sigma_{n,\rho}}
		=
		\sinh^2\rho\,g_{\Sph^{n-1}},
		\qquad
		K_{n,\rho}^{\Sigma}
		=
		\frac1{\sinh^2\rho}.
	\end{equation}

	Put
	\[
		A_{n,\rho}\coloneqq\frac{\sinh\rho}{\sqrt n},
		\qquad
		\mathsf S_n
		\coloneqq
		(\Sph^{n-1},\sqrt n\,|\theta-\eta|,\sigma_{n-1}).
	\]
	For $A>0$, define
\[ F_A(t)\coloneqq2\arsinh\frac{At}{2} \qquad(t\ge0). \]
	For an mm-space $(X,d_X,\mu_X)$, or simply $X$, put
	\[
		F_A(X)\coloneqq(X,F_A\circ d_X,\mu_X).
	\]
	Then
	\begin{equation}\label{eq:shell-transform-curvature}
		\Sigma_{n,\rho}=F_{A_{n,\rho}}(\mathsf S_n),
		\qquad
		A_{n,\rho}^{-2}=nK_{n,\rho}^{\Sigma}.
	\end{equation}
\end{prop}

\begin{proof}
	The hyperbolic polar distance identity is
	\begin{equation}\label{eq:hyperbolic-polar-distance}
		\cosh d_{\Hyp}\bigl((r,\theta),(t,\eta)\bigr)
		=
		\cosh(r-t)+\frac12\sinh r\sinh t\,|\theta-\eta|^2.
	\end{equation}
	Set $r=t=\rho$ and use
	$\cosh(2\arsinh u)=1+2u^2$ to obtain
	\eqref{eq:exact-shell-metric}.  The polar expression for the hyperbolic
	metric,
	\[
		dr^2+\sinh^2r\,g_{\Sph^{n-1}},
	\]
	gives \eqref{eq:induced-shell-metric}.  Substituting $A_{n,\rho}$ in
	\eqref{eq:exact-shell-metric} proves
	\eqref{eq:shell-transform-curvature}.  This completes the proof.
\end{proof}

\begin{prop}[Fixed finite samples become regular simplices]
\label{prop:equilateral}
	Suppose that $s_n\to0$, $s_n\rho_n\to r>0$, and
	$s_n(\mathbf R_n-\rho_n)\to0$ in probability.  For fixed $k\ge2$, let
	$\mathbf X_{n,1},\ldots,\mathbf X_{n,k}$ be independent points with
	distribution $\mu_n$.  Then their scaled distance matrices converge in
	probability to the distance matrix of a regular simplex with edge length
	$2r$:
	\[
		\bigl(s_nd_{\Hyp}(\mathbf X_{n,i},\mathbf X_{n,j})\bigr)_{i,j=1}^k
		\overset{\mathrm P}{\longrightarrow}
		\bigl(2r\,\boldsymbol{1}_{\{i\ne j\}}\bigr)_{i,j=1}^k.
	\]
\end{prop}

\begin{proof}
	Under the coupling in \Cref{prop:radial-coupling}, the scaled error
	produced by replacing each distance with its value on the shell of radius
	$\rho_n$ converges to zero in probability.  For independent uniform
	directions, independence and rotational invariance give
	\[
		\mathbb E\langle\Theta_{n,i},\Theta_{n,j}\rangle^2=\frac1n.
	\]
	Chebyshev's inequality gives
	$\langle\Theta_{n,i},\Theta_{n,j}\rangle\to0$ in probability.  Therefore,
	$|\Theta_{n,i}-\Theta_{n,j}|^2
	=2-2\langle\Theta_{n,i},\Theta_{n,j}\rangle$ implies
\[ |\Theta_{n,i}-\Theta_{n,j}| \overset{\mathrm P}{\longrightarrow}\sqrt2. \]
	The triangle inequality and \eqref{eq:exact-shell-metric} give
	\[
		2\log\bigl(\sinh\rho_n
		|\Theta_{n,i}-\Theta_{n,j}|\bigr)
		\le
		d_{n,\rho_n}(\Theta_{n,i},\Theta_{n,j})
		\le2\rho_n.
	\]
	After multiplication by $s_n$, both bounds converge in probability to
	$2r$.  Apply this argument simultaneously to the finitely many pairs.
	This completes the proof.
\end{proof}

To identify the shell pyramid, we need to track the scale on which
positive-mass subsets of a shell can be mutually separated.  For Borel sets
$B,C\subset Y$, put
\[ d_Y(B,C)\coloneqq\inf\{d_Y(x,y)\mid x\in B,\ y\in C\}. \]
For an mm-space $Y$ and positive numbers $\kappa_0,\ldots,\kappa_N$, define
its separation distance by
\[ \Sep(Y;\kappa_0,\ldots,\kappa_N) \coloneqq \sup_{B_0,\ldots,B_N}\min_{i\ne j}d_Y(B_i,B_j), \]
where the supremum ranges over Borel sets $B_i\subset Y$ satisfying
$\mu_Y(B_i)\ge\kappa_i$ for every $i$.

\begin{unnumberedlem}
	For every $N\ge1$ and every positive tuple
	$\kappa_0,\ldots,\kappa_N$ with $\sum_{i=0}^N\kappa_i<1$, there are
	constants $0<c\le C<+\infty$ such that
	\[
		c\le\Sep(\mathsf S_n;\kappa_0,\ldots,\kappa_N)\le C
	\]
	for all sufficiently large $n$.
\end{unnumberedlem}

\begin{proof}
	Let $\gamma^1$ be the standard Gaussian probability measure on $\R$.
	Since $\sum_i\kappa_i<1$, and since $\gamma^1$ is nonatomic and positive
	on every open interval, there are compact intervals
	$I_0,\ldots,I_N$ at positive mutual distances such that
	$\gamma^1(I_i)>\kappa_i$.  More explicitly, choose $\delta>0$ such that
	$\sum_i\kappa_i+(2N+1)\delta<1$.  The Gaussian distribution function
	allows us to arrange closed intervals of masses $\kappa_i+\delta$
	alternately with open gaps of mass $\delta$.

	Define $\pi_n\colon\Sph^{n-1}\to\R$ by
	$\pi_n(\theta)\coloneqq\sqrt n\,\theta_1$.  The Maxwell--Boltzmann
	distribution law \cite[Proposition~2.1]{shioya2016mmg} gives
	$(\pi_n)_*\sigma_{n-1}\to\gamma^1$ weakly.  The endpoints of every $I_i$
	are $\gamma^1$-null, and therefore
	$\sigma_{n-1}(\pi_n^{-1}(I_i))\ge\kappa_i$ for all sufficiently large
	$n$.  Moreover,
	\[
		|\pi_n(\theta)-\pi_n(\eta)|\le\sqrt n\,|\theta-\eta|.
	\]
	Thus, setting $c\coloneqq\min_{i\ne j}d_{\R}(I_i,I_j)>0$ gives the lower
	separation bound.

	For the upper bound, equip $\Sph^{n-1}$ with its spherical geodesic
	distance.  The first nonzero Laplacian eigenvalue of the unit sphere is
	$n-1$ \cite[Example~7.36]{shioya2016mmg}.  After multiplying the distance
	by $\sqrt n$, the first eigenvalue is $(n-1)/n$.  Chord distance does not
	exceed spherical geodesic distance, and removing entries from a tuple
	cannot decrease separation.  Applying
	\cite[Proposition~2.38]{shioya2016mmg} to two sets gives
\[ \Sep(\mathsf S_n;\kappa_0,\ldots,\kappa_N) \le\frac{2}{\sqrt{((n-1)/n)\min_i\kappa_i}}. \]
	The right-hand side is uniformly bounded for $n\ge2$.
	This completes the proof.
\end{proof}

\begin{thm}[Deterministic-shell convergence]
\label{thm:deterministic-shell-input}
	Let $\rho_n>0$ be deterministic.  If
	$\sinh\rho_n/\sqrt n\to+\infty$, then
	\begin{equation}\label{eq:deterministic-shell-input}
		\Pyr\left(
		\left(2\log\frac{\sinh\rho_n}{\sqrt n}\right)^{-1}
		\Sigma_{n,\rho_n}\right)
		\longrightarrow\Ple
	\end{equation}
	weakly.
\end{thm}

\begin{proof}
	We have $A_{n,\rho_n}\to+\infty$.  Fix a positive tuple
	$\kappa_0,\ldots,\kappa_N$ with $\sum_i\kappa_i<1$.  Since $F_A$ is
	continuous and increasing, the definition of separation and
	\eqref{eq:shell-transform-curvature} give
	\[
		\Sep\left(
		(2\log A_{n,\rho_n})^{-1}\Sigma_{n,\rho_n};
		\kappa_0,\ldots,\kappa_N\right)
		=
		\frac{F_{A_{n,\rho_n}}\bigl(\Sep(\mathsf S_n;\kappa_0,\ldots,\kappa_N)\bigr)}
		{2\log A_{n,\rho_n}}.
	\]
	Take the constants $c,C$ from the preceding lemma.  If $A\ge1$ and
	$c\le r\le C$, then
	$\arsinh u=\log(u+\sqrt{u^2+1})$ and
	$\sqrt{u^2+1}\le u+1$ yield
	\[
		2\log A+2\log(c/2)\le F_A(r)\le2\log A+2\log(C+1).
	\]
	Therefore,
	\begin{equation}\label{eq:normalized-shell-separation}
		\Sep\left(
		(2\log A_{n,\rho_n})^{-1}\Sigma_{n,\rho_n};
		\kappa_0,\ldots,\kappa_N\right)\longrightarrow1.
	\end{equation}

	By \cite[Proposition~8.5(1)]{shioya2016mmg},
	\eqref{eq:normalized-shell-separation} implies that every weak
	subsequential limit $\mathcal Q$ of the principal pyramids of the
	normalized shells contains $\Ple$.  Here the proposition is applied
	directly with $\delta=1$.

	We prove the reverse inclusion.  Suppose that $Y\in\mathcal Q$ and
	$\diam Y>1$.  Since $\mu_Y$ has full support and is inner regular, there
	are a number $\eta>0$ and positive-mass compact sets
	$K_0,K_1\subset Y$ such that
	\[
		d_Y(K_0,K_1)>1+4\eta.
	\]
	Indeed, take sufficiently small open balls about two points at distance
	greater than $1$, and choose a positive-mass compact subset of each ball.

	Weak convergence of the pyramids allows us, after passing to a further
	subsequence if necessary, to choose mm-spaces $Y_n$ satisfying
\[ Y_n\prec(2\log A_{n,\rho_n})^{-1}\Sigma_{n,\rho_n}, \qquad \Box(Y_n,Y)<\frac1n. \]
	By the parameter definition of the box distance
	\cite[Definition~4.4]{shioya2016mmg}, there are parameters
	$\varphi_n\colon[0,1)\to Y_n$ and $\psi_n\colon[0,1)\to Y$ and a Borel
	set $E_n\subset[0,1)$ such that
	\[
		\mathcal L^1(E_n)\ge1-\frac2n,\qquad
		|d_{Y_n}(\varphi_n(s),\varphi_n(t))-d_Y(\psi_n(s),\psi_n(t))|\le\frac2n
	\]
	for every $s,t\in E_n$.  Put
	$E_{i,n}\coloneqq E_n\cap\psi_n^{-1}(K_i)$.  Then
	\[
		\mathcal L^1(E_{i,n})\ge\mu_Y(K_i)-\frac2n.
	\]
	By Lusin's theorem, there is a compact set $C_{i,n}\subset E_{i,n}$ such
	that $\varphi_n|_{C_{i,n}}$ is continuous and
	\[
		\mathcal L^1(C_{i,n})\ge\mu_Y(K_i)-\frac3n.
	\]
	Put $K_{i,n}\coloneqq\varphi_n(C_{i,n})$. This set is compact, and the
	parameter identity for $\varphi_n$ gives
	\[
		\mu_{Y_n}(K_{i,n})
		=\mathcal L^1(\varphi_n^{-1}(K_{i,n}))
		\ge\mathcal L^1(C_{i,n})
		\ge\mu_Y(K_i)-\frac3n.
	\]
	The distortion bound gives
	\[
		d_{Y_n}(K_{0,n},K_{1,n})\ge d_Y(K_0,K_1)-\frac2n.
	\]
	It follows that, for all sufficiently large $n$,
	\[
		\Sep\left(Y_n;\frac{\mu_Y(K_0)}2,\frac{\mu_Y(K_1)}2\right)>1+\eta.
	\]
	Separation is monotone under domination
	\cite[Lemma~2.25]{shioya2016mmg}, so the same separation distance of the
	normalized shell is also greater than $1+\eta$.  The two compact sets
	are disjoint, and thus
	$\bigl(\mu_Y(K_0)+\mu_Y(K_1)\bigr)/2<1$.  This contradicts
	\eqref{eq:normalized-shell-separation}.  Therefore,
	$\mathcal Q\subset\Ple$, and hence $\mathcal Q=\Ple$.

	The weak topology on pyramids is compact and metrizable
	\cite[Theorem~6.22]{shioya2016mmg}.  Every subsequence has a weakly
	convergent further subsequence, whose limit must be $\Ple$ by the
	argument above.  The whole sequence therefore converges weakly to
	$\Ple$.  This completes the proof.
\end{proof}

The observable diameter in the collapsing regime satisfies the following
estimate from the same shell representation.  For
every $0<\kappa<1$, there is $C_\kappa>0$, independent of $\rho$, such that
\begin{equation}\label{eq:shell-obsdiam}
	\ObsDiam(\Sigma_{n,\rho};-\kappa)
	\le2F_{A_{n,\rho}}(C_\kappa).
\end{equation}
Indeed, chordal distance on $\Sph^{n-1}$ is at most spherical geodesic
distance.  It follows from
\cite[Lemma~2.25, Propositions~2.19 and~2.26, and
Theorem~2.21]{shioya2016mmg} that $C_\kappa>0$ can be chosen independently
of $n$ so that all Borel sets $B,C\subset\Sph^{n-1}$ with
$\sigma_{n-1}(B),\sigma_{n-1}(C)\ge\kappa/2$ satisfy
\[ \inf_{\theta\in B,\,\eta\in C}\sqrt n\,|\theta-\eta|\le C_\kappa. \]
The cited results give a uniform bound for all sufficiently large $n$, and
enlarging the constant accounts for the remaining finitely many dimensions.
Since $F_A$ is continuous and increasing, the distance between $B$ and $C$
for the transformed metric $F_A(\sqrt n\,|\theta-\eta|)$ is at most
$F_A(C_\kappa)$.  For any real-valued $1$-Lipschitz function with respect to
this transformed metric, choose its lower and upper $\kappa/2$ quantiles.
The two quantile sets each have measure at least $\kappa/2$, and the
difference of the quantiles is at most their set distance.  The interval
between them has measure at least $1-\kappa$.  Taking the supremum over the
functions and using \eqref{eq:shell-transform-curvature} proves
\eqref{eq:shell-obsdiam}.

For $c>0$ and a pyramid $\mathcal Q$, put
\[ c\mathcal Q\coloneqq\{cY\mid Y\in\mathcal Q\}. \]
In particular,
\[ c\mathcal P_{\le D}=\mathcal P_{\le cD}. \]

\begin{thm}[Continuity of positive metric rescaling]
\label{thm:rescaling-continuity}
	Suppose that $c_n\to c>0$ and that pyramids $\mathcal Q_n$ converge
	weakly to $\mathcal Q$.  Then
	\[
		c_n\mathcal Q_n\longrightarrow c\mathcal Q
	\]
	weakly.  In particular, if $\Pyr(Y_n)\to\mathcal Q$, then
	\[
		\Pyr(c_nY_n)\longrightarrow c\mathcal Q.
	\]
\end{thm}

\begin{proof}
	We first verify continuity for mm-spaces.  If $Y_n$ converges to $Y$ in
	box distance, then
\[ \Box(c_nY_n,c_nY) \le \max\{1,c_n\}\Box(Y_n,Y). \]
	For fixed $Y$ and $\eps>0$, choose a compact set $K\subset Y$ of measure
	at least $1-\eps$.  The diagonal coupling on $K$ gives
\[ \Box(c_nY,cY) \le \max\{\eps,|c_n-c|\diam K\}. \]
	Letting first $n\to+\infty$ and then $\eps\downarrow0$ proves that
	$c_nY_n$ converges to $cY$ in box distance.

	Use the sequential Painlev\'e--Kuratowski criterion for weak convergence
	of pyramids
	\cite[Definition~6.4 and Proposition~6.9]{shioya2016mmg}.  If
	$Y=cZ\in c\mathcal Q$, there are $Z_n\in\mathcal Q_n$ that converge to
	$Z$.  The preceding continuity gives $c_nZ_n\to Y$.  Conversely, suppose
	along a subsequence that $Y_n\in c_n\mathcal Q_n$ and $Y_n\to Y$.  Then
	$c_n^{-1}Y_n\in\mathcal Q_n$ and
	$c_n^{-1}Y_n\to c^{-1}Y$.  The upper-limit condition gives
	$c^{-1}Y\in\mathcal Q$, and therefore $Y\in c\mathcal Q$.  Finally,
	positive rescaling preserves domination, so
	$c_n\Pyr(Y_n)=\Pyr(c_nY_n)$.  This completes the proof.
\end{proof}

\begin{proof}[Proof of \Cref{thm:main}]
	By \Cref{prop:radial-coupling} and the stability of weak pyramid limits
	under vanishing box distance
	\cite[Proposition~5.5 and Theorem~6.25]{shioya2016mmg}, it is enough to
	work with $s_n\Sigma_{n,\rho_n}$.  If $a>0$, then
	$A_{n,\rho_n}\to+\infty$ and
	\[
		s_n\Sigma_{n,\rho_n}
		=
		(2a_{n,\rho_n})
		\left((2\log A_{n,\rho_n})^{-1}\Sigma_{n,\rho_n}\right).
	\]
	\Cref{thm:deterministic-shell-input,thm:rescaling-continuity} show that
	the pyramids converge weakly to $\mathcal P_{\le2a}$.

	Suppose that $a\le0$.  For $A,C>0$,
\[ F_A(C) \le2\log(1+AC) \le2\log(1+C)+2(\log A)_+. \]
	Thus, for every fixed $C>0$,
	\[
		s_nF_{A_{n,\rho_n}}(C)\longrightarrow0.
	\]
	By \eqref{eq:shell-obsdiam}, the sequence
	$(s_n\Sigma_{n,\rho_n})$ is a L\'evy family.  Its pyramids therefore
	converge weakly to the one-point pyramid $\mathcal P_{\le0}$
	\cite[Corollary~5.8 and Theorem~6.25]{shioya2016mmg}.

	For the signed-distance assertion, choose polar coordinates in which the
	first angular coordinate is normal to $H_n$.
	Then
	\begin{equation}\label{eq:signed-distance-formula}
		u_{H_n}(\rho,\theta)=\arsinh(\sinh\rho\,\theta_1).
	\end{equation}
	Fermi coordinates around $H_n$ show that $u_{H_n}$ is $1$-Lipschitz.  The
	random variables $\sqrt n\,\Theta_{n,1}$ converge in distribution to
	$N(0,1)$ \cite[Proposition~2.1]{shioya2016mmg}.  Their signs are symmetric
	two-point random variables independent of their absolute values.

	If $a>0$, then $A_{n,\rho_n}\to+\infty$ and
	\[
		s_n\left[
		\arsinh\bigl(A_{n,\rho_n}|\sqrt n\,\Theta_{n,1}|\bigr)
		-\log A_{n,\rho_n}
		-\log\bigl(2|\sqrt n\,\Theta_{n,1}|\bigr)\right]
		\overset{\mathrm P}{\longrightarrow}0.
	\]
	Moreover,
\[ s_n\log|\sqrt n\,\Theta_{n,1}| \overset{\mathrm P}{\longrightarrow}0. \]
	For the latter convergence, first restrict
	$|\sqrt n\,\Theta_{n,1}|$ to a fixed interval bounded away from zero and
	infinity, and then use convergence to a nonatomic Gaussian
	distribution.  It follows that
	$s_n|u_{H_n}(\rho_n,\Theta_n)|\to a$ in probability.

	If $a\le0$, fix $M>0$.  On
	$|\sqrt n\,\Theta_{n,1}|\le M$,
	\[
		0\le s_n|u_{H_n}(\rho_n,\Theta_n)|
		=\frac{s_n}{2}F_{A_{n,\rho_n}}\bigl(2|\sqrt n\,\Theta_{n,1}|\bigr)
		\le\frac{s_n}{2}F_{A_{n,\rho_n}}(2M)\longrightarrow0.
	\]
	The convergence of $\sqrt n\,\Theta_{n,1}$ to a Gaussian distribution
	also makes this sequence tight.  Letting first $n\to+\infty$ and then
	$M\to+\infty$ shows that
	$s_n|u_{H_n}(\rho_n,\Theta_n)|\to0$ in probability.  Finally,
	the radial coupling changes signed distance by at most
	$s_n|\mathbf R_n-\rho_n|$.  This proves
	\eqref{eq:general-signed-witness}.
	This completes the proof.
\end{proof}

\begin{cor}[Raw radial Dirac limit]\label{cor:raw-dirac}
	Suppose that $s_n\to0$ and $s_n\mathbf R_n\to r>0$ in probability.  Assume also
	that
	\[
		s_n\log n\longrightarrow\lambda\in[0,+\infty].
	\]
	Then
	\begin{equation}\label{eq:raw-dirac-limit}
		\Pyr(s_nX_n)\longrightarrow
		\mathcal P_{\le(2r-\lambda)_+}
	\end{equation}
	weakly.  If $\lambda<2r$, then signed distance from a hyperplane
	converges in distribution to
\[ \frac12\delta_{-r+\lambda/2} +\frac12\delta_{r-\lambda/2}. \]
	In particular, if $s_n\log n\to0$, then the radial Dirac radius, the
	effective pyramid radius, and half the limiting distance between fixed
	finite samples all equal $r$.
\end{cor}

\begin{proof}
	Put $\rho_n\coloneqq r/s_n$.  Then
	$s_n(\mathbf R_n-\rho_n)\to0$ in probability and
	\[
		a_{n,\rho_n}
		=
		r-\frac12s_n\log n-s_n\log2
		+s_n\log(1-e^{-2r/s_n})
		\longrightarrow r-\frac\lambda2.
	\]
	Apply \Cref{thm:main,prop:equilateral}.  This completes the proof.
\end{proof}

\section{Modal shells and tangential curvature}
\label{sec:modal-curvature}

Assume from now on that $\mu_n$ has a smooth radial density relative to
hyperbolic volume, with normalizing constant $Z_n$:
\begin{equation}\label{eq:radial-gibbs}
	d\mu_n(x)
	=
	Z_n^{-1}e^{-V_n(r(x))}
	\,d\operatorname{vol}_{\Hyp^n}(x),
	\qquad
	V_n\in C^2((0,+\infty)).
\end{equation}
The tangential eigenvalue below provides a curvature interpretation of the
effective shell radius already obtained, rather than driving the pyramid
limit.
The one-dimensional radial density is proportional to
\begin{equation}\label{eq:radial-density} p_n(r)\coloneqq e^{-V_n(r)}(\sinh r)^{n-1}. \end{equation}

\begin{prop}[Curvature identity at a radial mode]
\label{prop:modal-curvature}
	Let $m_n>0$ be an interior mode of $p_n$.  Define the Bakry--\'Emery
	Ricci tensor by
	\[
		\Ric_{V_n}\coloneqq\Ric_g+\Hess V_n,
	\]
	and put $g_T\coloneqq g-dr\otimes dr$.  Then
	\begin{equation}\label{eq:general-weighted-ricci}
		\Ric_{V_n}
		=
		\bigl(V_n''(r)-(n-1)\bigr)\,dr\otimes dr
		+
		\bigl(V_n'(r)\coth r-(n-1)\bigr)\,g_T.
	\end{equation}
	The modal equation is
\begin{equation}\label{eq:modal-equation} V_n'(m_n)=(n-1)\coth m_n. \end{equation}
	In particular, along the tangential directions of the modal shell,
	\begin{equation}\label{eq:modal-tangential-curvature}
		\left.\Ric_{V_n}\right|_{T\Sigma_{n,m_n}}
		=
		\frac{n-1}{\sinh^2m_n}
		\left.g\right|_{T\Sigma_{n,m_n}}.
	\end{equation}
	If $\lambda_n^{\mathrm{tan}}(m_n)$ denotes the tangential eigenvalue,
	then
	\begin{equation}\label{eq:curvature-matching}
		\lambda_n^{\mathrm{tan}}(m_n)
		=
		(n-1)K_{n,m_n}^{\Sigma}
		=
		\frac{n-1}{\sinh^2m_n}.
	\end{equation}
	For $n\ge3$, this identity can also be written as
\[ \left.\Ric_{V_n}\right|_{T\Sigma_{n,m_n}} = \frac{n-1}{n-2}\Ric_{\Sigma_{n,m_n}}. \]
\end{prop}

\begin{proof}
	Since $\Ric_g=-(n-1)g$ and
	\[
		\Hess r=\coth r\,(g-dr\otimes dr),
	\]
	the Hessian chain rule gives \eqref{eq:general-weighted-ricci}.  The
	logarithmic derivative of \eqref{eq:radial-density} vanishes at $m_n$,
	which proves \eqref{eq:modal-equation}.  Substituting this equation into
	the tangential coefficient gives
\[ V_n'(m_n)\coth m_n-(n-1) = (n-1)(\coth^2m_n-1) = \frac{n-1}{\sinh^2m_n}. \]
	The remaining identities follow from \Cref{prop:shell-geometry} and the
	Ricci curvature of a round $(n-1)$-sphere.  This completes the proof.
\end{proof}

\begin{thm}[Radial Dirac--curvature principle]
\label{thm:dirac-curvature-principle}
	Let $s_n\to0$, let $m_n$ be an interior mode of the radial density, and
	suppose that
	\begin{equation}\label{eq:modal-dirac}
		s_n(\mathbf R_n-m_n)\overset{\mathrm P}{\longrightarrow}0.
	\end{equation}
	Then
	\begin{equation}\label{eq:curvature-radius-identities}
		\begin{aligned}
a_{n,m_n} &=-\frac{s_n}{2}\log\bigl(nK_{n,m_n}^{\Sigma}\bigr)\\
&=-\frac{s_n}{2}\log\left( \frac n{n-1}\lambda_n^{\mathrm{tan}}(m_n) \right).
		\end{aligned}
	\end{equation}
	If $a_{n,m_n}\to a\in\R$, then
	\[
		\Pyr(s_nX_n)\longrightarrow\mathcal P_{\le2a_+}
	\]
	weakly, and signed distance from a hyperplane converges in distribution
	to
	\[
		\frac12\delta_{-a_+}+\frac12\delta_{a_+}.
	\]

	Suppose in addition that, for some $r>0$ and
	$\lambda\in[0,+\infty)$,
	\begin{equation}\label{eq:raw-modal-radius}
		s_nm_n\longrightarrow r,
		\qquad
		s_n\log n\longrightarrow\lambda.
	\end{equation}
	Then
	\begin{equation}\label{eq:three-radii}
		\begin{aligned}
r &= \lim_{n\to+\infty} \left(-\frac{s_n}{2}\log K_{n,m_n}^{\Sigma}\right),\\
		\left(r-\frac\lambda2\right)_+
		&=
		\lim_{n\to+\infty}
		\left[
		-\frac{s_n}{2}\log\left(
		\frac n{n-1}\lambda_n^{\mathrm{tan}}(m_n)
		\right)\right]_+.
		\end{aligned}
	\end{equation}
	In particular, if $s_n\log n\to0$, then the radial Dirac radius, the
	shell-curvature decay radius, the tangential-curvature decay radius, and
	the effective pyramid radius all equal $r$.
\end{thm}

\begin{proof}
	Equation \eqref{eq:curvature-radius-identities} follows from
	\Cref{prop:shell-geometry,prop:modal-curvature}.  Apply \Cref{thm:main}
	with $\rho_n=m_n$ to obtain the pyramid and signed-distance limits.
	Finally,
\[ s_n\log\sinh m_n = s_nm_n-s_n\log2+s_n\log(1-e^{-2m_n}) \longrightarrow r. \]
	Substitution of this limit and $s_n\log n\to\lambda$ in
	\eqref{eq:curvature-radius-identities} proves
	\eqref{eq:three-radii}.  This completes the proof.
\end{proof}

\begin{prop}[An exact zero-curvature shell near the modal shell]
\label{prop:zero-shell-transversality}
	Put
\begin{equation}\label{eq:F-zero} \Xi_n(r)\coloneqq V_n'(r)-(n-1)\tanh r. \end{equation}
	The tangential Bakry--\'Emery Ricci tensor vanishes on the shell of
	radius $\zeta>0$ if and only if $\Xi_n(\zeta)=0$.  Moreover,
\begin{equation}\label{eq:F-at-mode} \Xi_n(m_n)=\frac{2(n-1)}{\sinh(2m_n)}. \end{equation}

	Suppose that $c_n,\delta_n>0$ satisfy
	\begin{equation}\label{eq:transversality}
		\Xi_n'(r)\ge c_n
		\quad(m_n-\delta_n\le r\le m_n),
		\qquad
		\Xi_n(m_n)\le c_n\delta_n.
	\end{equation}
	Then there is a unique
	$\zeta_n\in[m_n-\delta_n,m_n]$ such that $\Xi_n(\zeta_n)=0$, and
	\begin{equation}\label{eq:zero-mode-gap-general}
		0\le m_n-\zeta_n
		\le
		\frac{2(n-1)}{c_n\sinh(2m_n)}.
	\end{equation}
	Let $s_n>0$ with $s_n\to0$.  In particular,
	\begin{equation}\label{eq:zero-mode-scaled-gap}
		\frac{s_n(n-1)}{c_n\sinh(2m_n)}
		\longrightarrow0
	\end{equation}
	implies $s_n(m_n-\zeta_n)\to0$.  If also
	$\liminf_{n\to+\infty}\zeta_n>0$, then
	\[
		a_{n,m_n}-a_{n,\zeta_n}\longrightarrow0.
	\]
	If, in addition,
	\[
		s_n(\mathbf R_n-m_n)\overset{\mathrm P}{\longrightarrow}0,
	\]
	then the modal and exact zero-curvature shells have the same radial
	Dirac, pyramid, and signed-distance limits.
\end{prop}

\begin{proof}
	The tangential coefficient in \eqref{eq:general-weighted-ricci} gives the
	zero-curvature equation.  The modal equation gives
\[ \Xi_n(m_n) = (n-1)(\coth m_n-\tanh m_n) = \frac{2(n-1)}{\sinh(2m_n)}. \]
	By \eqref{eq:transversality},
\[ \Xi_n(m_n-\delta_n) \le \Xi_n(m_n)-c_n\delta_n \le0 < \Xi_n(m_n). \]
	The intermediate value theorem and strict monotonicity on the interval
	give the unique zero $\zeta_n$.  The mean value theorem proves
	\eqref{eq:zero-mode-gap-general}.

	Under the last assumption, $\coth\zeta_n$ is uniformly bounded for all
	large $n$.  Applying the mean value theorem to $\log\sinh r$ gives
\[ 0\le a_{n,m_n}-a_{n,\zeta_n} \le s_n(m_n-\zeta_n)\coth\zeta_n \longrightarrow0. \]
	The scaled difference between the radial centers also converges to zero,
	so the remaining assertions follow from \Cref{thm:main}.
	This completes the proof.
\end{proof}

\begin{rem}
	This exponential proximity is a feature of the Gaussian potential
	$V(r)=r^2/2\sigma^2$, not a general phenomenon for radial densities on
	$\mathbb H^n$: curvature plays no role in the convergence itself, which
	only requires a radius sequence within the same tolerance of the mode.
	We record the zero-curvature shell nonetheless because it is a curious
	feature of this specific example: although $\mathbb H^n$ has constant
	curvature $-1$, the weighted curvature associated with the Gaussian
	potential nearly vanishes exactly where the concentration takes place.
\end{rem}

\section{The two hyperbolic Gaussians}\label{sec:gaussians}

Fix $\sigma>0$.  With $Z^{\mathrm R}_{n,\sigma}$ as the normalizing constant,
the Riemannian Gaussian is defined by
\begin{equation}\label{eq:riemannian-gaussian}
	d\mu^{\mathrm R}_{n,\sigma}(x)
	\coloneqq
	\frac1{Z^{\mathrm R}_{n,\sigma}}
	\exp\left(-\frac{r(x)^2}{2\sigma^2}\right)
	d\operatorname{vol}_{\Hyp^n}(x).
\end{equation}
The wrapped Gaussian $\mu^{\mathrm W}_{n,\sigma}$ is the pushforward under
$\exp_o$ of
$N(0,\sigma^2I_n)$ on $T_o\Hyp^n$.  Its radial potential relative to
hyperbolic volume is
\begin{equation}\label{eq:wrapped-potential}
	V^{\mathrm W}_{n,\sigma}(r)
	\coloneqq
	\frac{r^2}{2\sigma^2}
	+(n-1)\log\frac{\sinh r}{r}.
\end{equation}
The quotient $\sinh r/r$ is interpreted through its smooth extension at the
origin.

Work in statistics and machine learning on hyperbolic Gaussian and
wrapped-normal-type distributions is primarily concerned with modeling,
sampling, and inference.  A foundational construction in this direction is the
wrapped normal distribution on hyperbolic space
\cite{nagano2019wrapped}.  Recent work includes finite mixture modeling with
Riemannian Gaussian distributions and profile-likelihood inference for
anisotropic hyperbolic wrapped normal models
\cite{you2026mixture,you2026profile}.  The question here is different: we
study high-dimensional metric-measure limits and phase-transition scales for
these distributions.

Write $\RGauss{n}{\sigma}$ for the Riemannian Gaussian mm-space defined by
\eqref{eq:riemannian-gaussian}, and write $\WGauss{n}{\sigma}$ for the wrapped
Gaussian mm-space.  The latter notation records the pushforward construction
by $\exp_o$, whereas the former gives a separate name to the space defined
intrinsically by a density against hyperbolic volume.  In both names the
subscript is the parameter of the Gaussian density, which in Euclidean
geometry is the variance.  Let
$\mathbf X^{\mathrm R}_{n,\sigma}$ and
$\mathbf X^{\mathrm W}_{n,\sigma}$ have distributions
$\mu^{\mathrm R}_{n,\sigma}$ and
$\mu^{\mathrm W}_{n,\sigma}$, respectively, and put
\[
	\mathbf R^{\mathrm R}_{n,\sigma}
	\coloneqq r(\mathbf X^{\mathrm R}_{n,\sigma}),
	\qquad
	\mathbf R^{\mathrm W}_{n,\sigma}
	\coloneqq r(\mathbf X^{\mathrm W}_{n,\sigma}).
\]
The hyperbolic polar volume element is
\[ d\operatorname{vol}_{\Hyp^n} = (\sinh r)^{n-1}\,dr\,d\omega_{n-1}. \]
For the wrapped Gaussian, the exponential-map Jacobian cancels this factor,
leaving $r^{n-1}$ in the radial density.

\begin{thm}[Radial limits for the Gaussians]
\label{thm:gaussian-radial-laws}
	Put
	\[
		\bar r^{\mathrm R}_{n,\sigma}
		\coloneqq(n-1)\sigma^2,
		\qquad
		\bar r^{\mathrm W}_{n,\sigma}
		\coloneqq\sigma\sqrt n.
	\]
	As $n\to+\infty$,
	\begin{align}
		\mathbf R^{\mathrm R}_{n,\sigma}
		-\bar r^{\mathrm R}_{n,\sigma}
		&\longrightarrow N(0,\sigma^2)
		&&\text{in total variation},
		\label{eq:R-radial-law}\\
		\mathbf R^{\mathrm W}_{n,\sigma}
		-\bar r^{\mathrm W}_{n,\sigma}
		&\overset{\mathrm d}{\longrightarrow}N(0,\sigma^2/2).
		\label{eq:W-radial-law}
	\end{align}
	Moreover, $\mathbf R^{\mathrm W}_{n,\sigma}/\sigma$ has exactly the
	chi distribution with $n$ degrees of freedom.
\end{thm}

\begin{proof}
	The radial density of the Riemannian Gaussian is proportional to
	\[
		\exp\left(-\frac{r^2}{2\sigma^2}\right)(\sinh r)^{n-1}.
	\]
	Put $y=r-\bar r^{\mathrm R}_{n,\sigma}$ and complete the square.  The
	density of the centered radius is proportional to the density of
	$N(0,\sigma^2)$ multiplied by
	\[
		\boldsymbol{1}_{\{y>-\bar r^{\mathrm R}_{n,\sigma}\}}
		\left(1-e^{-2(\bar r^{\mathrm R}_{n,\sigma}+y)}\right)^{n-1}.
	\]
	This factor lies between zero and one and converges to one for every
	$y\in\R$.  The dominated convergence theorem shows that its integral
	also converges to one.  After normalization, the densities converge in
	$L^1$ to the Gaussian density, which proves
	\eqref{eq:R-radial-law}.

	For the wrapped Gaussian, the exponential map preserves distance from
	the origin.  Its radius is therefore the Euclidean norm of a vector with
	distribution $N(0,\sigma^2I_n)$.  This proves the chi representation.
	Applying the central limit theorem to the square of the radius gives
	\[
		\frac{(\mathbf R^{\mathrm W}_{n,\sigma}/\sigma)^2-n}{\sqrt{2n}}
		\overset{\mathrm d}{\longrightarrow}N(0,1),
		\qquad
		\frac{\mathbf R^{\mathrm W}_{n,\sigma}}{\sigma\sqrt n}
		\overset{\mathrm P}{\longrightarrow}1.
	\]
	The identity
	\[
		\mathbf R^{\mathrm W}_{n,\sigma}-\sigma\sqrt n
		=
		\frac{\sigma\sqrt{2n}}
		{\mathbf R^{\mathrm W}_{n,\sigma}/\sigma+\sqrt n}
		\frac{(\mathbf R^{\mathrm W}_{n,\sigma}/\sigma)^2-n}{\sqrt{2n}}
	\]
	and Slutsky's theorem prove \eqref{eq:W-radial-law}.
	This completes the proof.
\end{proof}

Let $m^{\mathrm R}_{n,\sigma}$ and $m^{\mathrm W}_{n,\sigma}$ denote the
modal radii.  The modal equation in \Cref{prop:modal-curvature} gives
\begin{equation}\label{eq:gaussian-modes}
	m^{\mathrm R}_{n,\sigma}
	=
	(n-1)\sigma^2\coth m^{\mathrm R}_{n,\sigma},
	\qquad
	m^{\mathrm W}_{n,\sigma}
	=
	\sigma\sqrt{n-1}.
\end{equation}
Consequently,
\begin{equation}\label{eq:mode-scales}
	\frac{m^{\mathrm R}_{n,\sigma}}n\longrightarrow\sigma^2,
	\qquad
	\frac{m^{\mathrm W}_{n,\sigma}}{\sqrt n}\longrightarrow\sigma.
\end{equation}
The centers in \Cref{thm:gaussian-radial-laws} may be replaced by the
corresponding modal radii.  Let
$\lambda^{\mathrm R}_{\mathrm{tan}}(r)$ and
$\lambda^{\mathrm W}_{\mathrm{tan}}(r)$ denote the tangential eigenvalues of
the two Bakry--\'Emery Ricci tensors.

\begin{thm}[Exact tangential zero-curvature shells for the Gaussians]
\label{thm:gaussian-zero-shells}
	For all sufficiently large $n$, there are unique radii
	$\zeta^{\mathrm R}_{n,\sigma}$ and
	$\zeta^{\mathrm W}_{n,\sigma}$ on which the corresponding tangential
	Bakry--\'Emery Ricci tensor vanishes exactly.

	For the Riemannian Gaussian,
	\begin{equation}\label{eq:R-zero-equation}
		\zeta^{\mathrm R}_{n,\sigma}
		\coth\zeta^{\mathrm R}_{n,\sigma}
		=
		(n-1)\sigma^2,
	\end{equation}
	and
	\begin{equation}\label{eq:R-zero-sandwich}
		\zeta^{\mathrm R}_{n,\sigma}
		<
		\bar r^{\mathrm R}_{n,\sigma}
		<
		m^{\mathrm R}_{n,\sigma}.
	\end{equation}
	The two gaps satisfy
	\begin{align}
		\frac{\bar r^{\mathrm R}_{n,\sigma}
		-\zeta^{\mathrm R}_{n,\sigma}}
		{2\bar r^{\mathrm R}_{n,\sigma}
		e^{-2\bar r^{\mathrm R}_{n,\sigma}}}
		&\longrightarrow1,
		\label{eq:R-zero-gap}\\
		\frac{m^{\mathrm R}_{n,\sigma}
		-\bar r^{\mathrm R}_{n,\sigma}}
		{2\bar r^{\mathrm R}_{n,\sigma}
		e^{-2\bar r^{\mathrm R}_{n,\sigma}}}
		&\longrightarrow1.
		\label{eq:R-mode-gap}
	\end{align}

	For the wrapped Gaussian,
	\begin{equation}\label{eq:W-zero-equation}
		\frac{(\zeta^{\mathrm W}_{n,\sigma})^2}
		{\sigma^2(n-1)}
		=
		1-\frac{2\zeta^{\mathrm W}_{n,\sigma}}
		{\sinh(2\zeta^{\mathrm W}_{n,\sigma})},
	\end{equation}
	and
	\begin{equation}\label{eq:W-zero-gap}
		0<m^{\mathrm W}_{n,\sigma}-\zeta^{\mathrm W}_{n,\sigma},
		\qquad
		\frac{m^{\mathrm W}_{n,\sigma}
		-\zeta^{\mathrm W}_{n,\sigma}}
		{2(m^{\mathrm W}_{n,\sigma})^2
		e^{-2m^{\mathrm W}_{n,\sigma}}}
		\longrightarrow1.
	\end{equation}
	In particular,
	\begin{equation}\label{eq:zero-scales}
		\frac{\zeta^{\mathrm R}_{n,\sigma}}n\longrightarrow\sigma^2,
		\qquad
		\frac{\zeta^{\mathrm W}_{n,\sigma}}{\sqrt n}\longrightarrow\sigma.
	\end{equation}
	The centers in \Cref{thm:gaussian-radial-laws} may also be replaced by
	the exact zero-curvature radii.
\end{thm}

\begin{proof}
	For the Riemannian Gaussian, the potential is
	\[
		\frac{r^2}{2\sigma^2}.
	\]
	Formula \eqref{eq:general-weighted-ricci} gives the zero-curvature equation
	\eqref{eq:R-zero-equation}.  The function $r\coth r$ increases strictly
	from $1$ to $+\infty$ on $(0,+\infty)$, so the positive solution is
	unique for large $n$.  The modal equation and
	$\tanh r<1<\coth r$ give \eqref{eq:R-zero-sandwich}.  The exact gap
	identities
	\[
		\bar r^{\mathrm R}_{n,\sigma}
		-\zeta^{\mathrm R}_{n,\sigma}
		=
		\frac{2\zeta^{\mathrm R}_{n,\sigma}}
		{e^{2\zeta^{\mathrm R}_{n,\sigma}}-1},
		\qquad
		m^{\mathrm R}_{n,\sigma}
		-\bar r^{\mathrm R}_{n,\sigma}
		=
		\frac{2m^{\mathrm R}_{n,\sigma}}
		{e^{2m^{\mathrm R}_{n,\sigma}}+1}
	\]
	prove \eqref{eq:R-zero-gap} and \eqref{eq:R-mode-gap}.

	Differentiating the wrapped potential gives
	\[
		\lambda^{\mathrm W}_{\mathrm{tan}}(r)
		=
		\frac{r\coth r}{\sigma^2}
		+(n-1)\left(
		\frac1{\sinh^2r}-\frac{\coth r}{r}\right).
	\]
	Its zero equation is equivalent to \eqref{eq:W-zero-equation}.  The
	function
\[ r\longmapsto \frac{1-2r/\sinh(2r)}{r^2} \]
	decreases strictly from $2/3$ to zero on $(0,+\infty)$.  To verify the
	sign of its derivative, put $t=2r$ and use
\[ 2\sinh^2t-t\sinh t-t^2\cosh t = \sum_{k=3}^{\infty} \frac{4^k-4k^2}{(2k)!}t^{2k}>0. \]
	Thus \eqref{eq:W-zero-equation} has a unique positive solution for all
	large $n$.  Moreover,
	\[
		m^{\mathrm W}_{n,\sigma}-\zeta^{\mathrm W}_{n,\sigma}
		=
		\frac{
		2(m^{\mathrm W}_{n,\sigma})^2
		\zeta^{\mathrm W}_{n,\sigma}}
		{(m^{\mathrm W}_{n,\sigma}
		+\zeta^{\mathrm W}_{n,\sigma})
		\sinh(2\zeta^{\mathrm W}_{n,\sigma})}.
	\]
	This proves \eqref{eq:W-zero-gap}.  The scale limits and centered radial
	limits now follow from \Cref{thm:gaussian-radial-laws}.
	This completes the proof.
\end{proof}

\begin{cor}[Coincidence of four radii for the Gaussians]
\label{cor:gaussian-coincidence}
	The modal radius, exact zero-curvature radius, shell-curvature decay
	radius, and tangential-curvature decay radius satisfy
	\begin{align}
		\frac{m^{\mathrm R}_{n,\sigma}}n,\quad
		\frac{\zeta^{\mathrm R}_{n,\sigma}}n,\quad
		-\frac1{2n}\log K_{n,m^{\mathrm R}_{n,\sigma}}^\Sigma,\quad
		-\frac1{2n}\log\left(
		\frac n{n-1}
		\lambda^{\mathrm R}_{\mathrm{tan}}
		(m^{\mathrm R}_{n,\sigma})\right)
		&\longrightarrow\sigma^2,
		\label{eq:R-radius-coincidence}\\
		\frac{m^{\mathrm W}_{n,\sigma}}{\sqrt n},\quad
		\frac{\zeta^{\mathrm W}_{n,\sigma}}{\sqrt n},\quad
		-\frac1{2\sqrt n}
		\log K_{n,m^{\mathrm W}_{n,\sigma}}^\Sigma,\quad
		-\frac1{2\sqrt n}\log\left(
		\frac n{n-1}
		\lambda^{\mathrm W}_{\mathrm{tan}}
		(m^{\mathrm W}_{n,\sigma})\right)
		&\longrightarrow\sigma.
		\label{eq:W-radius-coincidence}
	\end{align}
	Moreover,
	\begin{align}
		\Box\left(
		n^{-1}\RGauss{n}{\sigma},
		n^{-1}\Sigma_{n,\zeta^{\mathrm R}_{n,\sigma}}
		\right)&\longrightarrow0,
		\label{eq:R-zero-box}\\
		\Box\left(
		n^{-1/2}\WGauss{n}{\sigma},
		n^{-1/2}\Sigma_{n,\zeta^{\mathrm W}_{n,\sigma}}
		\right)&\longrightarrow0.
		\label{eq:W-zero-box}
	\end{align}
	Therefore,
	\begin{align}
		\Pyr(n^{-1}\RGauss{n}{\sigma})
		&\longrightarrow\mathcal P_{\le2\sigma^2},
		\label{eq:R-pyramid}\\
\Pyr(n^{-1/2}\WGauss{n}{\sigma}) &\longrightarrow\mathcal P_{\le2\sigma}. \label{eq:W-pyramid}
	\end{align}
	For any choice of totally geodesic hyperplanes
	$H_n\subset\Hyp^n$ through $o$, let $u_{H_n}$ be signed distance from
	$H_n$, positive on one chosen side.  Then
	\begin{align}
		(n^{-1}u_{H_n})_*\mu^{\mathrm R}_{n,\sigma}
		&\overset{\mathrm d}{\longrightarrow}
		\frac12\delta_{-\sigma^2}+\frac12\delta_{\sigma^2},
		\label{eq:R-witness}\\
		(n^{-1/2}u_{H_n})_*\mu^{\mathrm W}_{n,\sigma}
		&\overset{\mathrm d}{\longrightarrow}
		\frac12\delta_{-\sigma}+\frac12\delta_{\sigma}.
		\label{eq:W-witness}
	\end{align}
	At the respective critical scales, every fixed finite sample also
	converges to a regular-simplex distance matrix with edge length
	$2\sigma^2$ or $2\sigma$.
\end{cor}

\begin{proof}
	The curvature-radius limits follow from
	\Cref{thm:dirac-curvature-principle,thm:gaussian-zero-shells}.  By
	\Cref{thm:gaussian-radial-laws,thm:gaussian-zero-shells},
	\[
		\frac{\mathbf R^{\mathrm R}_{n,\sigma}
		-\zeta^{\mathrm R}_{n,\sigma}}n
		\overset{\mathrm P}{\longrightarrow}0,
		\qquad
		\frac{\mathbf R^{\mathrm W}_{n,\sigma}
		-\zeta^{\mathrm W}_{n,\sigma}}{\sqrt n}
		\overset{\mathrm P}{\longrightarrow}0.
	\]
	\Cref{prop:radial-coupling} gives \eqref{eq:R-zero-box} and
	\eqref{eq:W-zero-box}.  Applying \Cref{thm:main} at scales $n^{-1}$ and
	$n^{-1/2}$ proves the pyramid and signed-distance limits.  The
	fixed-sample assertion follows from \Cref{prop:equilateral}.
	This completes the proof.
\end{proof}

\section{Critical limits and the phase transition property}
\label{sec:critical-limits}

For a Borel probability measure $\nu$ on a metric space and
$0<\kappa<1$, put
\[ \PartDiam(\nu;1-\kappa) \coloneqq \inf\{\diam B\mid B\text{ is Borel and }\nu(B)\ge1-\kappa\}. \]
This partial diameter is written $\operatorname{diam}(\nu;1-\kappa)$ in
\cite{shioya2016mmg}.  The observable diameter of an mm-space $Y$ is
\[ \ObsDiam(Y;-\kappa) \coloneqq \sup_f\PartDiam(f_*\mu_Y;1-\kappa), \]
where $f$ ranges over all $1$-Lipschitz maps from $Y$ to $\R$.  A sequence
$(Y_n)$ is a L\'evy family if
$\ObsDiam(Y_n;-\kappa)\to0$ for every $\kappa>0$.

$\ObsDiam(Y;-\kappa)\coloneqq0$ for $\kappa\ge1$.  For a pyramid
$\mathcal Q$, put
\[ \ObsDiam(\mathcal Q;-\kappa) \coloneqq \sup_{Y\in\mathcal Q}\ObsDiam(Y;-\kappa). \]
The published limit formula \cite[Theorem~1.1]{yokota2024obsdiam} states
that if pyramids $\mathcal Q_n$ converge weakly to $\mathcal Q$ and
$\kappa>0$, then
\[
\begin{aligned}
	\ObsDiam(\mathcal Q;-\kappa)
	&=\lim_{\eps\to0+}\liminf_{n\to+\infty}\ObsDiam(\mathcal Q_n;-(\kappa+\eps))\\
	&=\lim_{\eps\to0+}\limsup_{n\to+\infty}\ObsDiam(\mathcal Q_n;-(\kappa+\eps)).
\end{aligned}
\]
Observable diameter is nonincreasing in the mass parameter, so the inner
limit superior is nonincreasing in $\eps$ and the outer limit is its
supremum.  Therefore, for every $\kappa,\eps>0$,
\[ \limsup_{n\to+\infty}\ObsDiam(\mathcal Q_n;-(\kappa+\eps)) \le\ObsDiam(\mathcal Q;-\kappa). \]
This inequality supplies the uniform observable-diameter bound in the
subcritical regime.

We use the separation distance introduced above.  The sequence $(Y_n)$ infinitely dissipates if
\[ \Sep(Y_n;\kappa_0,\ldots,\kappa_N)\longrightarrow+\infty \]
for every $N\ge1$ and every positive tuple with
$\sum_{i=0}^N\kappa_i<1$.

A sequence of mm-spaces $(Y_n)$ has the \emph{phase transition property} with
critical scale $t_n>0$ if the following two conditions hold for every positive
sequence $(\alpha_n)$.  If $\alpha_n/t_n\to0$, then $(\alpha_nY_n)$ is a
L\'evy family.  If $\alpha_n/t_n\to+\infty$, then
$(\alpha_nY_n)$ infinitely dissipates.

\begin{thm}[Criterion for the phase transition property]
\label{thm:critical-pyramid-ptp}
	Let $(t_n)$ be positive and suppose that
	\[
		\Pyr(t_nY_n)\longrightarrow\mathcal Q
	\]
	weakly.  Assume that
\[ \ObsDiam(\mathcal Q;-\kappa)<+\infty \qquad(\kappa>0), \]
	and that, for every $N\ge1$ and every positive tuple
	$\kappa_0,\ldots,\kappa_N$ satisfying $\sum_i\kappa_i<1$, there are
	$Y\in\mathcal Q$ and $\eps>0$ such that
	\[
		\Sep(Y;\kappa_0+\eps,\ldots,\kappa_N+\eps)>0.
	\]
	Then $(Y_n)$ has the phase transition property with critical scale
	$t_n$.
\end{thm}

\begin{proof}
	By monotonicity of observable diameter under domination
	\cite[Proposition~2.18(3)]{shioya2016mmg}, the observable
	diameter of $\Pyr(t_nY_n)$ equals that of $t_nY_n$.  For every $0<\kappa<1$,
	monotonicity in $\kappa$ and the published limit formula
	\cite[Theorem~1.1]{yokota2024obsdiam}, applied with parameter $\kappa/2$,
	give
	\[
	\begin{aligned}
\limsup_{n\to\infty}\ObsDiam(t_nY_n;-\kappa) &\le\limsup_{n\to\infty}\ObsDiam(t_nY_n;-3\kappa/4)\\
		&\le\ObsDiam(\mathcal Q;-\kappa/2)<+\infty.
	\end{aligned}
	\]
	If $\alpha_n/t_n\to0$, homogeneity of observable diameter
	\cite[Proposition~2.19]{shioya2016mmg} gives
\[ \ObsDiam(\alpha_nY_n;-\kappa) =\frac{\alpha_n}{t_n}\ObsDiam(t_nY_n;-\kappa)\longrightarrow0. \]
	For $\kappa\ge1$, the observable diameter is zero by convention.  Thus
	$(\alpha_nY_n)$ is a L\'evy family.

	Fix $N\ge1$ and a positive tuple $\kappa_0,\ldots,\kappa_N$ satisfying
	$\sum_i\kappa_i<1$.  Choose $Y\in\mathcal Q$ and $\eps>0$ from the
	hypothesis, and take
	\[
		0<\delta<\Sep(Y;\kappa_0+\eps,\ldots,\kappa_N+\eps).
	\]
	The definition of separation distance and inner regularity of the
	probability measure give compact sets $K_0,\ldots,K_N\subset Y$ such that
\[ \mu_Y(K_i)>\kappa_i+\frac{\eps}{2}, \qquad d_Y(K_i,K_j)>\delta\quad(i\ne j). \]

	Since weak convergence of pyramids means sequential
	Painlev\'e--Kuratowski convergence, there are $W_n\in\Pyr(t_nY_n)$ such that
	$\Box(W_n,Y)\to0$.  Choose $\eta_n\to0$, couplings
	$\pi_n\in\mathcal T(\mu_Y,\mu_{W_n})$, and closed relations
	$R_n\subset Y\times W_n$ such that
	\[
		1-\pi_n(R_n)\le\eta_n,\qquad\operatorname{dis}R_n\le\eta_n.
	\]
	For each $i$, put
\[ B_{i,n}\coloneqq \{w\in W_n\mid (y,w)\in R_n\text{ for some }y\in K_i\}. \]
	The set $R_n\cap(K_i\times W_n)$ is closed in $K_i\times W_n$.  Since
	$K_i$ is compact, its projection $B_{i,n}$ onto $W_n$ is closed.  For
	all sufficiently large $n$,
\[ \mu_{W_n}(B_{i,n}) \ge\pi_n\bigl(R_n\cap(K_i\times W_n)\bigr) \ge\mu_Y(K_i)-\eta_n>\kappa_i. \]
	If $i\ne j$, then
\[ d_{W_n}(B_{i,n},B_{j,n}) \ge d_Y(K_i,K_j)-\operatorname{dis}R_n>\delta-\eta_n. \]
	Therefore,
	\[
		\liminf_{n\to\infty}\Sep(W_n;\kappa_0,\ldots,\kappa_N)\ge\delta.
	\]
	Since $W_n\prec t_nY_n$, monotonicity of separation under domination
	\cite[Lemma~2.25]{shioya2016mmg} gives the same lower bound for $t_nY_n$.
	If $\alpha_n/t_n\to+\infty$, homogeneity of separation
	\cite[Proposition~2.9]{ozawa2015limit} gives
	\[
		\Sep(\alpha_nY_n;\kappa_0,\ldots,\kappa_N)
		=\frac{\alpha_n}{t_n}\Sep(t_nY_n;\kappa_0,\ldots,\kappa_N)
		\longrightarrow+\infty.
	\]
	Thus $(\alpha_nY_n)$ infinitely dissipates.  These are precisely the two
	conditions in the definition of the phase transition property.  This
	completes the proof.
\end{proof}

\begin{lem}[Admissibility of diameter-bounded pyramids]
\label{lem:diameter-pyramid-admissible}
	For every $D>0$,
\[ \ObsDiam(\mathcal P_{\le D};-\kappa)<+\infty \qquad(\kappa>0). \]
	Moreover, for every $N\ge1$ and every positive tuple
	$\kappa_0,\ldots,\kappa_N$ satisfying $\sum_i\kappa_i<1$, there are
	$Y\in\mathcal P_{\le D}$ and $\eps>0$ such that
	\[
		\Sep(Y;\kappa_0+\eps,\ldots,\kappa_N+\eps)>0.
	\]
\end{lem}

\begin{proof}
	Every $Y\in\mathcal P_{\le D}$ has diameter at most $D$.  Therefore,
	\[
		\ObsDiam(\mathcal P_{\le D};-\kappa)\le D.
	\]

	Fix $N\ge1$ and a positive tuple
	$\kappa_0,\ldots,\kappa_N$ with sum less than one.  Choose $\eps>0$ so
	that
	\[
		\sum_{i=0}^N(\kappa_i+\eps)<1.
	\]
	On $\{0,\ldots,N,N+1\}$, take the metric for which all distinct points
	have distance $D$.  Give point $i$ mass $\kappa_i+\eps$ for
	$0\le i\le N$, and give the last point the remaining mass.  This
	mm-space belongs to $\mathcal P_{\le D}$, and the first $N+1$
	singletons show that
	\[
		\Sep(Y;\kappa_0+\eps,\ldots,\kappa_N+\eps)=D.
	\]
	This completes the proof.
\end{proof}

\begin{cor}[Phase transition property for the two Gaussians]
\label{thm:gaussian-ptp}
	The sequences $(\RGauss{n}{\sigma})$ and
	$(\WGauss{n}{\sigma})$ have the phase transition property with
	critical scales $n^{-1}$ and $n^{-1/2}$, respectively.
\end{cor}

\begin{proof}
	The limits in \eqref{eq:R-pyramid} and \eqref{eq:W-pyramid} are positive
	rescalings of diameter-bounded pyramids.  By
	\Cref{lem:diameter-pyramid-admissible}, the criterion in
	\Cref{thm:critical-pyramid-ptp} applies to both sequences.
	This completes the proof.
\end{proof}

For each fixed $\sigma>0$, both implications in the definition above have converses for the two families $(\RGauss{n}{\sigma})$ and $(\WGauss{n}{\sigma})$ under every positive rescaling sequence. At a finite positive subsequential ratio to the respective critical scale, \Cref{cor:gaussian-coincidence,thm:main} give a finite-diameter pyramid limit, which excludes infinite dissipation, and a nondegenerate two-point signed-distance observable, which excludes the L\'evy property, while the endpoint ratios $0$ and $+\infty$ are covered by \Cref{thm:gaussian-ptp}.

\begin{rem}[What becomes flat and what vanishes exactly]
\label{rem:what-is-flat}
	The sectional curvature of the ambient hyperbolic space remains $-1$.
	The intrinsic sectional curvature of a modal shell is positive, and the
	tangential Bakry--\'Emery Ricci eigenvalue is $(n-1)$ times that
	curvature.  The radii $\zeta^{\mathrm R}_{n,\sigma}$ and
	$\zeta^{\mathrm W}_{n,\sigma}$ determine different shells on which the
	tangential Bakry--\'Emery Ricci tensor vanishes exactly.  These shells
	approach the corresponding modal shells at exponential rates.

	In the Poincar\'e ball, the Euclidean radii of the exact zero-curvature
	shells approach the boundary according to
	\[
		\frac{1-\tanh(\zeta^{\mathrm R}_{n,\sigma}/2)}
		{2e^{-\bar r^{\mathrm R}_{n,\sigma}}}
		\longrightarrow1
	\]
	and
	\[
		\frac{1-\tanh(\zeta^{\mathrm W}_{n,\sigma}/2)}
		{2e^{-m^{\mathrm W}_{n,\sigma}}}
		\longrightarrow1.
	\]
	On either exact zero-curvature shell, the radial weighted-Ricci
	eigenvalue equals the full trace, and division by $n$ gives a limit of
	$-1$.  Exact tangential vanishing therefore does not imply vanishing of
	the ambient curvature or of the full weighted-Ricci tensor.
\end{rem}

\section*{Acknowledgments}

The author would like to thank Professor Takashi Shioya for many helpful
suggestions and guidance.  The author used Claude, GPT-5.5, and GPT-5.6-series
Codex models as AI-assisted tools in preparing this manuscript.  The author
reviewed and revised the mathematical content and takes full responsibility
for the final text.

\end{document}